\RequirePackage{fix-cm}
\documentclass[smallcondensed]{svjour3hack}     
\smartqed  

\journalname{Journal of Global Optimization}

\usepackage{amsmath}
\usepackage{amssymb}
\usepackage{graphicx}
\usepackage{hyperref}
\usepackage{footmisc}
\usepackage{color}
\usepackage{caption}
\usepackage{pdfpages}
\usepackage{bigstrut}
\usepackage{tensor}

\spnewtheorem{thm}[theorem]{Theorem}{\bfseries}{\normalfont}
\spnewtheorem{cor}[theorem]{Corollary}{\bfseries}{\normalfont}
\spnewtheorem{prop}[theorem]{Proposition}{\bfseries}{\normalfont}
\spnewtheorem{lem}[theorem]{Lemma}{\bfseries}{\normalfont}
\spnewtheorem{exa}[theorem]{Example}{\bfseries}{\normalfont}

\numberwithin{equation}{section}

\allowdisplaybreaks

\begin{document}

\title{Virtuous smoothing for global optimization}%

\author{Jon Lee \and Daphne Skipper}

\institute{Jon Lee \at IOE Dept., Univ. of Michigan. Ann Arbor, Michigan, USA.\\  \email{jonxlee@umich.edu} \and Daphne Skipper \at Dept. of Mathematics, U.S. Naval Academy. Annapolis, Maryland, USA.\\
\email{skipper@usna.edu}}

\date{\today}%

\maketitle

%

\begin{abstract}
In the context of global optimization
and mixed-integer non-linear programming,
generalizing a technique of D'Ambrosio, Fampa, Lee and Vigerske  for handling the square-root function,
we develop a virtuous smoothing method, using cubics, aimed at functions having some limited non-smoothness.
Our results pertain to root functions ($w^p$ with $0<p<1$) and their increasing concave relatives.
We provide (i) a sufficient condition (which applies to functions more general than root functions) for our
smoothing to be increasing and concave, (ii) a proof that when $p=1/q$ for integers $q\geq 2$, our smoothing
lower bounds the root function,  (iii) substantial progress (i.e., a proof for integers $2\leq q\leq 10,000$)
on the conjecture that our smoothing is a sharper bound on the root function than the natural and simpler
``shifted root function'', and (iv) for all root functions,
a quantification
of the superiority (in an average sense)
of our smoothing versus the shifted root function near~0.
\end{abstract}


\section*{Introduction}\label{sec:intro}
Important models for framing and attacking hard (often non-linear) combin\-a\-tor\-ial-optimization
problems are GO (global optimization)
and MINLP (mixed-integer non-linear programming).
Virtually all GO and MINLP solvers  (e.g.,
\verb;SCIP; \cite{Achterberg2009},
\verb;Baron; \cite{sahinidis},
\verb;Couenne; \cite{Belotti09},
\verb;Antigone; \cite{misener-floudas:ANTIGONE:2014})
apply some variant of spatial branch-and-bound (see, \cite{Smith99}, for example), and they rely on NLP (non-linear-programming) solvers, both to solve continuous relaxations (to generate lower bounds for minimization) and often to
generate good feasible solutions (to generate upper bounds).
Sometimes models are organized to have a convex relaxation,
and then either outer approximation, NLP-based branch-and-bound, or some hybrid of the two is employed (e.g., \verb;Bonmin; \cite{Bonami}; also see \cite{BLLW}). In such a case, NLP solvers are also heavily relied upon, and for the same uses as in the non-convex case.

Convergence of most NLP solvers (e.g. \verb;Ipopt; \cite{WB06}) requires that functions be twice continuously differentiable. Yet many models naturally utilize functions with some limited non-differentia\-bil\-i\-ty.
One approach to handle limited non-differentia\-bil\-i\-ty is smoothing.

Of course there is a vast literature on global optimization concerning convexification (see \cite{TawarSahinBook}). Such research aims at developing tractable
lower bounds for non-convex formulations $z:=\min \{f_0(x) ~:~ x\in \mathcal{F}\}$. Even when the only non-convexities are
integrality of some variables, improving the trivial convexification (relaxing integrality) is crucial to the success of algorithms such as outer-approximation (see \cite{Bonami}, for example). But lower bounding $z$ is not the complete story.
For example, spatial branch-and-bound (for formulations that are
non-convex after relaxing integrality) and outer-approximation
(for formulations that are convex after relaxing integrality)
both require solutions of the NLPs (convex or not) obtained by relaxing integrality.
They do this in an effort to find actual (incumbent) feasible solutions
and hence upper bounds on $z$. The success of this step,
requires close approximation of the MINLP and tractability of the NLPs.
Our results for smoothing are aimed at getting tractable NLPs in the presence of functions with limited non-differentiability.
For univariate concave functions, we aim for concave under-estimation
which has the effect of forcing the approximation to be near the
function that we approximate. Note that convexifying a univariate concave function via secant under-estimation can do a rather poor job of
close approximation. This idea of using one version of a function for
convexification in the context of lower bounding $z$ and another version
of a function aimed at coming closer to the MINLP  was implemented for objective functions in \verb;Bonmin;,
in the context of an application; see \cite{BDLLT06,BDLLT12} (a study of optimal water-network refurbishment via MINLP). Additionally, it is possible to simply use our smoothing methodology as a formulation pre-processor for a spatial branch-and-bound algorithm. With such a use,
the smoothing inherits the convexification structure of the MINLP; that is,
our approximation, because of its concavity, allows for secant under-estimation and tangent over-estimation, and this has been implemented in \verb;SCIP; (see \S\ref{subsec:soft}). Furthermore, in the context of
smoothing as a pre-processing step, \cite{DFLV2014,DFLV2015} used the
under-estimation property of our smoothing to get an \emph{a priori}
upper-bound on how much the optimal value of their smoothed MINLP
could be below $z$.
Finally,
also employing smooth concave under-estimators, \cite{Sergeyev1998} describes a global-optimization algorithm for univariate functions;
so we can see another use of such under estimators in global optimization.

Additionally in \cite{BDLLT06,BDLLT12},
an ad-hoc smoothing method is used to address non-differentiability near 0
of the Hazen-Williams (empirical) formula for the pressure drop
of turbulent water flow in a pipe as a function of the flow.
Choosing a small positive $\delta$ and fitting an odd homogeneous
quintic on $[-\delta,\delta]$, so as to match the function and its
first and second derivatives at $\pm\delta$ and the function value at 0,
the resulting piecewise function is smooth enough for NLP solvers.
However on $(-\delta,\delta)$ the quintic is neither an upper bound nor a lower bound
on the function it approximates.

In \cite{GMS13} (a study of the TSP with ``neighborhoods''), $\sqrt{w}$
is smoothed (near 0) by choosing again a small positive $\delta$ and then using a linear extrapolation of
$\sqrt{w}$ at $w=\delta$ to approximate $\sqrt{w}$ on $[0,\delta)$. Shortcomings of this
approach are that the resulting piecewise function is: (i) not twice differentiable at $w=\delta$, and
(ii) over-estimates $\sqrt{w}$ on $[0,\delta)$. Regarding (ii), in many formulations (see
\cite{DFLV2014,DFLV2015}, for example), we need
an under-estimate of the function we are approximating to get a valid relaxation.

To address the identified shortcomings of the methodology of  \cite{GMS13}
for smoothing square roots, motivated by developing
tractable mixed-integer non-linear-optimization models for the Euclidean Steiner Problem (see \cite{ITOR:ITOR12207}),
\cite{DFLV2014,DFLV2015} developed a virtuous method.
On the interval $[0,\delta]$, they fit a homogeneous
cubic, to match the function value of the
square root at the endpoints, and matching the
first and second derivatives at $w=\delta$.
They demonstrate that the resulting smooth function,
though piecewise defined, is increasing, strictly concave,
under-estimates $\sqrt{w}$ on $(0,\delta)$, and is
a stronger under-estimate of  $\sqrt{w}$ than
the more elementary ``shift'' $\sqrt{w+\lambda}-\sqrt{\lambda}$, for some
$\lambda>0$.
To make a fair comparison, $\delta$ is chosen as a
function of $\lambda$ so that the two approximating
functions have the same derivative at 0.
This makes sense because, for each approximation,
we would choose the value of the smoothing parameter
($\delta$ or $\lambda$) as small as the numerics
will tolerate --- that is, we would have an upper bound
for the derivative of the approximation (at zero, where
it is greatest).

For the  $\sqrt{\cdot}$ function, we have depicted all of these smoothings in Figure \ref{fig:smooth}.
From top to bottom:
The ``linear extrapolation''  follows the dot-dashed line ($\cdot\,\text{-}\cdot\text{-}\,\cdot$) below $\delta=0.1$.
The solid curve ($\textbf{-----}$) is the true $\sqrt{\cdot}$.
The ``smooth under-estimation'', which we advocate, follows the dotted curve ($\cdot\!\cdot\!\cdot\!\cdot\!\cdot\!\cdot$) below $\delta=0.1$.
The ``shift'',   chosen to have the same derivative as our preferred smoothing at 0, follows the weaker under-estimate (on the entire non-negative domain) given by the dashed curve (- - -).

\begin{figure}[ht]
\begin{center}
\includegraphics[width=.80\textwidth]{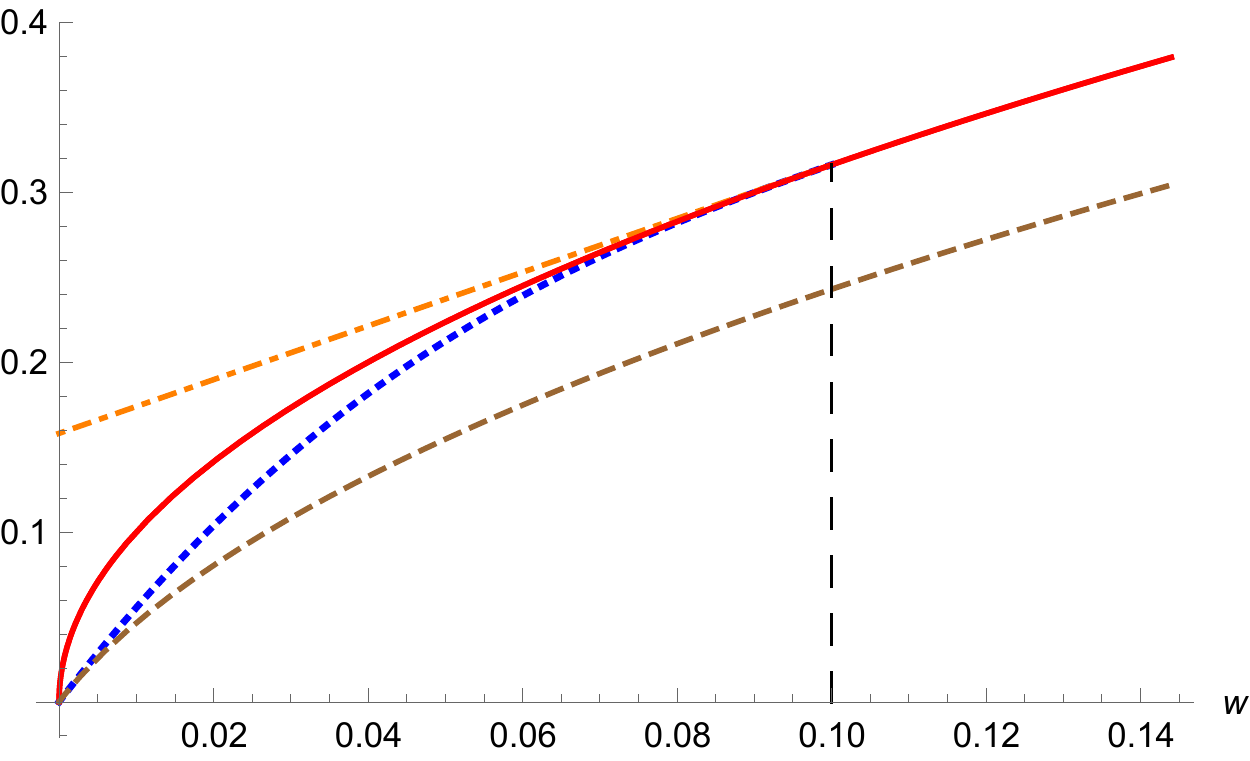}
\caption{Behavior of all smoothings of the square root}
\label{fig:smooth}
\end{center}
\end{figure}

Most of our results are for root functions ($w^p$ with $0<p<1$) and their increasing concave relatives. Smoothing roots (not just square roots) can play an important role
in working with $\ell_q$  norms ($q>1$) and quasi-norms ($0<q<1$):
$\ell_q(x):=(\sum_j x_j^q)^{1/q}$. Depending on $q$ (less than 1, or greater than 1), each inner power or
the outer power is a root function. Of course the use of norms is quite
common. Additionally, in sparse optimization,
$\ell_q$ quasi-norms are used to induce sparsity; see, for example, \cite{Wotao} and the references therein. So, our work can be used to apply global-optimization techniques in this important setting.
Additionally, root functions are natural for
fitting nearly-smooth functions to data that follows an increasing concave trend; for example, cost functions with economies of scale,
and the well-known Cobb-Douglas production function (and generalizations), relating production to labor and capital inputs, where the exponents of the inputs are the output elasticities (see \cite{douglas} and \cite{arrow}).
After such a data-analysis step, fitted functions can be
incorporated into optimization models, and our results would then be applicable; also see \cite{Cozad} for a modern data-driven integrated function-fitting/optimization methodology.
Additionally, roots occur in other signomial functions
besides the Cobb-Douglas production function (see \cite{Duffin1973}).
Finally, besides root functions, there are other simple univariate building-block
functions that our scheme applies to; for example $\log(1+w)$ and
${\rm ArcSinh}(\sqrt{w})$ (see Examples \ref{example:log} and \ref{example:sinh}).

In \S\ref{sec:general}, we provide a sufficient condition (which applies to functions more general than root functions) for our smoothing to be increasing and concave. Moreover, we give an interesting example to illustrate that when our condition is not satisfied, the conclusion need not hold.
In \S\ref{sec:roots}, we establish that when $p=1/q$ for integers $q\geq 2$, our smoothing lower bounds the root function. Having such control over the root function is important in the context of global optimization --- in fact, this was a key motivation of \cite{DFLV2014,DFLV2015}.
In \S\ref{sec:better}, we present substantial progress (i.e., a proof for integers $2\leq q\leq 10,000$) on the conjecture that our smoothing is a sharper bound on the root function than the natural and simpler ``shifted root function''.
In \S\ref{sec:avg} we quantify the average relative performance of our
smoothing and of the shifted root function near 0. We demonstrate
that our smoothing is much better with respect to this performance measure.
Finally, in \S\ref{sec:future}, we make some concluding remarks: describing alternatives, available software, some extended use, and our ongoing work.

\section{General smoothing via a homogenous cubic}\label{sec:general}

\subsection{Construction of our smoothing}\label{sec:cubic}
We are given a function $f$ defined on $[0, +\infty)$ having the following properties:
$f(0)=0$,  $f$ is increasing and concave on $[0, +\infty)$,
$f'(w)$ and $f''(w)$ are defined on all of $(0,+\infty)$, but $f'(0)$ is undefined.
For example, the \emph{root function} $f(w):=w^p$, with $0<p<1$, has these properties.
Our goal is to find a function $g$ that mimics $f$ well, but is
differentiable everywhere (in particular at 0). In addition, because our context
is global optimization, we want $g$ to lower bound $f$ on $[0, +\infty)$.
In this way, we can develop smooth relaxations of certain optimization problems involving $f$.

The definition of our function  $g$ depends on a parameter $\delta>0$.
Our function $g$ is simply $f$ on $[\delta,+\infty)$.
This parameter $\delta$ allows us to control the
derivative of $g$ at 0. Essentially, lowering $\delta$ increases the
derivative of $g$ at 0, and so in practice, we choose $\delta$ as low as the
numerics will tolerate.

We extend $g$ to $[0,\delta]$, as a homogeneous cubic, so that $g(0)=f(0)=0$, $g(\delta)=f(\delta)$, $g'(\delta)=f'(\delta)$ and $g''(\delta)=f''(\delta)$.
The homogeneity  immediately gives us $g(0)=0$, and
such a polynomial is the lowest-degree one that allows us to
 match $f$, $f'$ and $f''$ at $\delta$.
We choose the three coefficients of $g(w) := Aw^3 + Bw^2 + Cw$
so that the remaining three conditions are satisfied.

The constants $A$, $B$ and $C$ are solutions to the system:
\begin{alignat*}{3}
(g(\delta)=)\quad& \delta^3 A + \delta^2 B + \delta C && = f(\delta)\\
(g'(\delta)=)\quad& 3\delta^2 A + 2\delta B + C && = f'(\delta)\\
(g''(\delta)=)\quad& 6\delta A + 2B  && = f''(\delta).
\end{alignat*}

\noindent We find that
\begin{eqnarray*}
A &~=~& \frac{f(\delta)}{\delta^3}~-~\frac{f'(\delta)}{\delta^2}~+~\frac{f''(\delta)}{2\delta}~, \\
B &~=~& -~\frac{3f(\delta)}{\delta^2} ~+~ \frac{3f'(\delta)}{\delta} ~-~ f''(\delta)~, \\
C &~=~& \frac{3f(\delta)}{\delta} ~-~ 2f'(\delta) ~+~ \frac{\delta f''(\delta)}{2}~.
\end{eqnarray*}

By construction, we have the following result.

\begin{prop}
The constructed function $g$ has
$g(0)=0$, $g(\delta)=f(\delta)$, $g'(\delta)=f'(\delta)$ and $g''(\delta)=f''(\delta)$.
\end{prop}

\subsection{Increasing and concave}\label{sec:properties}

Mimicking $f$ should mean that $g$ is increasing and concave on all of $[0,+\infty)$.
Next, we give a sufficient condition for this. The condition is a bound on the amount of negative
curvature of $f$ at $\delta$.

\begin{thm}\label{thm:inc_conv}
Let $\delta>0$ be given.
On $[\delta,+\infty)$, let $f$ be increasing and differentiable, with $f'$ non-increasing (decreasing).
Let $f(0)=0$, and let $f$ be twice differentiable at $\delta$.
If
\[
f''(\delta) \geq \frac{2}{\delta}\left( f'(\delta)-\frac{f(\delta)}{\delta}\right), \tag{$T_\delta$}\label{Tdelta}
\]
the associated function $g$ is increasing and concave (strictly concave) on  $[0,+\infty)$.
\end{thm}
\begin{proof}
For $w\in [0,\delta]$, the third derivative of $g$ is the constant
\[
g'''(w) = \frac{6}{\delta^3}\left(f(\delta) - \delta f'(\delta) + \frac{\delta^2}{2}f''(\delta)\right).
\]
The first factor is clearly positive.  The inequality requirement on $f''(\delta)$ in our hypothesis makes the second factor non-negative.  We conclude that the third derivative of $g$ is non-negative, implying that the second derivative of $g$ is non-decreasing to a non-positive (negative) value, $g''(\delta)=f''(\delta)$, on the interval $[0,\delta]$.  Consequently, $g'(w)$ is non-increasing (decreasing) to $g'(\delta)=f'(\delta)>0.$

Note that the assumptions on $f$ imply that for $w\in [\delta,+\infty)$, $g'(w)=f'(w)$ is non-increasing (decreasing) and
$g'(w)=f'(w)>0$. Therefore, $g$ is concave (strictly concave) and increasing on $[0, +\infty)$.
\qed
\end{proof}

Root functions, that is power functions of the form
$f(w)=w^p$ with $0<p<1$, fit our general framework:
$f(0)=0$,  $f$ is increasing and concave on $[0, +\infty)$,
$f'(w)$ and $f''(w)$ are defined on all of $(0,+\infty)$, but $f'(0)$ is undefined.
Indeed, our work was inspired by the construction for $p=1/2$ in \cite{DFLV2014,DFLV2015}.
Next, we verify that Theorem \ref{thm:inc_conv} applies to root functions.

\begin{lem}
For $f(w)=w^p$, $0<p<1$, we have that $f$ satisfies \ref{Tdelta} for all $\delta>0$.
\end{lem}

\begin{proof}
For $f(w)=w^p$, the inequality \ref{Tdelta} is
\[
p(p-1)\delta^{p-2} \geq \frac{2}{\delta}\left(p\delta^{p-1}-\frac{\delta^p}{\delta}\right),
\]
which simplifies to
$(p-1)(p-2) \geq 0$, and which is satisfied because $p<1$
\qed
\end{proof}

So, by Theorem \ref{thm:inc_conv}, we have the following result.
\begin{cor} \label{cor:pow_concave}
For $f(w)=w^p$, $0<p<1$, the associated $g$ is increasing and strictly concave on $[0,+\infty)$.
\end{cor}

The following very useful fact is easy to see.
\begin{lem}\label{lemma_cone}
The set of $f$ with domain $[\delta,+\infty)$ satisfying any of
\begin{itemize}
\item $f(0)=0$,
\item $f$ is increasing,
\item $f$ is differentiable,
\item $f'$ is non-increasing or decreasing,
\item $f$ is twice differentiable at $\delta$,
\item \ref{Tdelta}
\end{itemize}
is a (blunt) cone in function space.
\end{lem}

As a consequence of Lemma \ref{lemma_cone} and Corollary \ref{cor:pow_concave},
adding a root function to any  $f$ that is differentiable at 0 and satisfies \emph{all} of
the properties listed in Lemma \ref{lemma_cone},
we get such a function that is  non-differentiable at 0
and has decreasing first derivative.

Next, we give a couple of natural examples to demonstrate that Theorem \ref{thm:inc_conv}
applies to other functions besides root functions.

\begin{exa}\label{example:log}
Let $f(w):=\log(1+w)$, which is clearly concave and increasing on $[0,+\infty)$, and has $f(0)=0$.
To verify that \ref{Tdelta} is satisfied for $\delta > 0$, we consider the expression $f''(\delta)-\frac{2}{\delta}\left(f'(\delta) - \frac{f(\delta)}{\delta}\right)$, which simplifies to
\[
\frac{2(1+\delta)^2 \log(1+\delta) -3\delta^2-2\delta}{\delta^2(1+\delta)^2}.
\]
The denominator of this expression is positive so we focus on the numerator, which we define to be $k(\delta)$.  The second derivative of the numerator, $k''(\delta)=4\log(1+\delta)$, is positive for $\delta > 0$, implying that the $k'(\delta)=4(1+\delta)\log(1+\delta)-6\delta$ increases from $k'(0)=0$.  Therefore, $k(\delta)$ likewise increases from $k(0)=0$.  We conclude that $T_\delta$ is satisfied for $\delta > 0$.
Note that by Lemma \ref{lemma_cone}, we can add $\sqrt{w}$ to $f$ to get an example
that is not differentiable at 0.   \phantom{a}\hfill $\diamondsuit$
\end{exa}

\begin{exa}\label{example:sinh}
Let $f(w):={\rm ArcSinh}(\sqrt{w})=\log(\sqrt{w}+\sqrt{1+w})$ on $[0,+\infty)$. Clearly $f(0)=0$.
We have $f'(w)=1/(2 \sqrt{w} \sqrt{w+1})$, which is non-negative on $(0,+\infty)$, so $f$ is increasing, but it is not differentiable at 0.
Additionally, $f''(x)= (-2 x-1)/(4 x^{3/2} (x+1)^{3/2})$, which is clearly negative on $(0,+\infty)$, so $f$ is strictly concave.

To verify that \ref{Tdelta} is satisfied for $\delta > 0$, we consider the expression $f''(\delta)-\frac{2}{\delta}\left(f'(\delta) - \frac{f(\delta)}{\delta}\right)$, which simplifies to
\[
\frac{8 (\delta+1)^{3/2} {\rm ArcSinh}\left(\sqrt{\delta}\right)-\sqrt{\delta} (6 \delta+5)}{4 \delta^2
   (\delta+1)^{3/2}}.
\]
The denominator of this expression is positive so we focus on the numerator, which we define to be $k(\delta)$.
We have that $k(0)=0$, so we will be able to conclude that $k$ is
non-negative if we can show that it is non-decreasing.
Note that this $k$ is not concave, so we cannot follow the method of the previous example.
Rather,
we calculate
\[
k'(\delta)=\frac{3}{2 \sqrt{\delta}}-5 \sqrt{\delta}+12 \sqrt{\delta+1} \log
   \left(\sqrt{\delta}+\sqrt{\delta+1}\right).
\]
We will seek to demonstrate $k'(\delta)\geq 0$ by showing
 $k'(\delta)-3/2 \sqrt{\delta}\geq 0$.
 The derivative of $k'(\delta)-3/2 \sqrt{\delta}$ is
 \[
 \frac{7}{2 \sqrt{\delta}}+\frac{6\,  {\rm ArcSinh}\left(\sqrt{\delta}\right)}{\sqrt{\delta+1}},
 \]
 which is clearly non-negative, and so
  \ref{Tdelta} is satisfied for $\delta > 0$.
  \phantom{a}\hfill $\diamondsuit$
  \end{exa}

It is natural to wonder whether $T_\delta$ is really needed in Theorem \ref{thm:inc_conv}.
Next, we give an example, where all conditions of Theorem \ref{thm:inc_conv} hold, except for
$T_\delta$, and the conclusion of Theorem \ref{thm:inc_conv} does \emph{not} hold.

\begin{exa}
For $\epsilon>0$,
let
\[
f(w):= \left\{
         \begin{array}{ll}
           \sqrt{w-1}-\sqrt{\epsilon}+\frac{1+\epsilon}{2\sqrt{\epsilon}}, & w\geq 1+\epsilon; \\
           \frac{1}{2\sqrt{\epsilon}}w, & w\leq 1+\epsilon.
         \end{array}
       \right.
\]
This function $f$, solid in Figure \ref{fig:graph_convex_g},  has $f(0)=0$, is differentiable, increasing and concave on $[0,+\infty)$ and
is twice differentiable for $w>1$.
\begin{figure}[ht!]
\centering
\includegraphics[width=.80\textwidth]{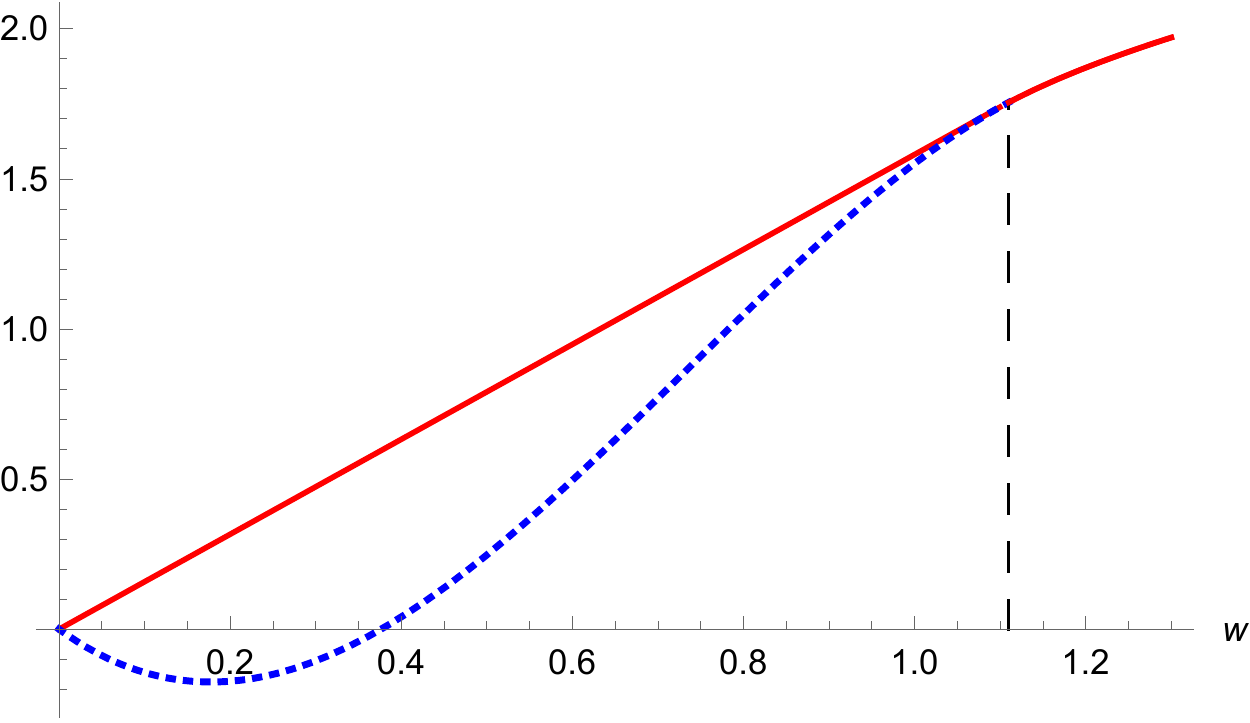}
\caption{$g$ is convex and decreasing near 0}
\label{fig:graph_convex_g}
\end{figure}
\vspace{12pt}

Now, let $\delta = 1 + \epsilon + \phi$, for $\phi>0$.  For $\epsilon = 1/10$ and $\phi = 1/100$, $f''(\delta)-\frac{2}{\delta}\left(f'(\delta) - \frac{f(\delta)}{\delta}\right) \approx -6.7$, so our condition $T_\delta$ is \emph{not} satisfied.  In fact, a few calculations reveal that the associated cubic $g$,  dotted in Figure \ref{fig:graph_convex_g},  is \emph{convex and decreasing} for $0 < w < \epsilon$.
The issue is that $f$ has too much negative curvature at $\delta$ to be concave on $[0,\delta]$ and have $f(0)=0$. Note that by Lemma \ref{lemma_cone}, we can add a small positive multiple of $\sqrt{w}$ to $f$ to get an example
that is strictly concave and not differentiable at 0.
\hfill $\diamondsuit$
\end{exa}


\section{Lower bound for roots}\label{sec:roots}

%
%

For root functions, applying our general construction,
direct calculation gives us the coefficients of $g(w) := Aw^3 + Bw^2 + Cw$~:
\begin{eqnarray*}
A &=& \delta^{p-3}(p^2-3p+2)/2,\\
B &=& -\delta^{p-2}(p^2-4p+3),\\
C &=& \delta^{p-1}(p^2-5p+6)/2.
\end{eqnarray*}


For use in global optimization, we want control over the relationship between  $f(w)$ and $g(w)$.
We believe that $g(w)\leq f(w)$ on $[0,\delta]$ for all root functions $f$. For now, we can only establish this for
$f(w):=w^p$, with $p:=1/q$ for \emph{integer} $q \geq 2$.
The case of $q=2$ was established in \cite{DFLV2014,DFLV2015}
via a much easier argument (see the proof of Part 5 of Theorem 1 in the Appendix of \cite{DFLV2014}).

\begin{thm}\label{thm:lower}
For $f(w):=w^p$, with $p=1/q$ for \emph{integer} $q \geq 2$, we have
$g(w) \leq f(w)$ for all $w \in [0,+\infty)$
\end{thm}

\begin{proof}
Clearly we can confine our attention to $[0,\delta]$.
 Our strategy is to express $f-g$ as the product of positive factors.  It is convenient to make several substitutions before we factor.  Starting with
\[
(f-g)(w) = w^p - \frac{\delta^{p-3}}{2}(p^2-3p+2)w^3 + \delta^{p-2}(p^2-4p+3)w^2 - \frac{\delta^{p-1}}{2}(p^2-5p+6)w,
\]
for $0 \leq w \leq \delta$, we introduce the change the variables
$t:=w^p$, $q:=1/p$ and $L:=\delta^p$ to arrive at
%
\[
 t - \frac{1}{2L^{3q-1}}\left ( \frac{1}{q^2}-\frac{3}{q}+2\right)t^{3q}+\frac{1}{L^{2q-1}}\left (\frac{1}{q^2}-\frac{4}{q}+3\right )t^{2q} - \frac{1}{2L^{q-1}}\left (\frac{1}{q^2}-\frac{5}{q}+6\right )t^{q},
\]
for $0 \leq t \leq L$.

We begin factoring by removing a positive monomial,
\begin{eqnarray*}
&=& \frac{t}{2q^2 L^{3q-1}}\left ( \phantom{\biggl(} 2q^2 L^{3q-1} - (6q^2-5q+1) L^{2q}t^{q-1} \right. \\
&& \left. + (6q^2-8q+2)L^{q}t^{2q-1} - (2q^2-3q+1)t^{3q-1} \phantom{\biggl(} \right ),
\end{eqnarray*}
so we can restrict our focus to the expression on the right, which we further simplify by making one last series of substitutions:
\begin{eqnarray*}
a &=& 2q^2,\\
b &=& 6q^2-5q+1,\\
c &=& 6q^2-8q+2,\\
d &=& 2q^2-3q+1.
\end{eqnarray*}
We claim this last expression,
\begin{equation*} \tag{$P_q$}
a L^{3q-1} - b L^{2q}t^{q-1}+cL^q t^{2q-1}-d t^{3q-1},
\end{equation*}
factors into $Q_q (L-t)^3$, where $Q_q$ is a polynomial in $L$ and $t$.  Because $0 \leq t \leq L$, $(L-t)^3$ is positive.  With the  Lemmas \ref{lem:terms} and \ref{app:lemma2} in the Appendix, we show that $Q_q$ is  positive for integer $q \geq 2$, implying that $g(w) \leq f(w)$ for $w \in [0,\delta]$ as desired.
\qed
\end{proof}

\section{Better bound}\label{sec:better}
We return, temporarily, to our general setting, where we are given a function $f$ defined over the interval $[0, +\infty)$ having the following properties:
$f(0)=0$,  $f$ is increasing and concave on $[0, +\infty)$,
$f'(w)$ and $f''(w)$ are defined on all of $(0,+\infty)$, but $f'(0)$ is undefined.

A natural and simple lower bound on $f$ is
to choose $\lambda>0$, and define the  \emph{shifted} $f$
as $h(w):= f(w+\lambda)-f(\lambda)$.
It is easy to see that
\[
h(w):= f(w+\lambda)-f(\lambda) \leq f(w),
\]
because $f$ is concave and non-negative at 0, which implies that
$f$ is subadditive on $[0,+\infty)$.

On the interval $[0,\delta]$, we wish to compare this $h$ (the shifted $f$)
to our smoothing $g$.
But $g$ is defined based on a choice of $\delta$ and
$h$ is defined based on a choice of $\lambda$, a
fair comparison is achieved by making these choices
so that the derivative at 0 is the same. In this way,
both smoothings of $f$ have the same numerical properties:
they both have the same maximum derivative (maximized
at zero where $f'$ blows up).

At $w=0$, the first derivative of $g$ is
\[
g'(0)= 3 f(\delta)/\delta - 2f'(\delta) +\delta f''(\delta)/2.
\]
We have that
\[
h'(0)=f'(\lambda).
\]
For each $\delta>0$, there is a $\lambda>0$ so that
$g'(0)=h'(0)$.
Now, suppose that $f'$ is decreasing on $[0,+\infty)$.
Then $(f')^{-1}$ exists, and
\[
\hat{\lambda} := (f')^{-1}
\left(
3 f(\delta)/\delta - 2f'(\delta) +\delta f''(\delta)/2
\right)
\]
is the value of $\lambda$ for which $g'(0)=h'(0)$.

So, in general, we want to check that for each $\delta>0$,
\[
f(w+\hat{\lambda})-f(\hat{\lambda})\leq g(w),
\tag{$*$}\label{star}
\]
for all $w\in (0,\delta)$. To go further, we now confine our
attention, once again, to root functions.

Already, \cite{DFLV2014,DFLV2015} established this for the square-root function,
though their proof has a certain weakness (see the proof of Proposition 3 in the Appendix of \cite{DFLV2014}),
relying on some numerics,
which our proof does not suffer from.
Our goal is to establish this property for all root functions.
This seems to be quite difficult, and so we set our focus now
on root functions of the form $f(w):=w^p$, with $p:=1/q$ for
\emph{integer} $q \geq 2$. We have a substantial partial result,
which as a by product provides an air-tight proof of the previous
result of \cite{DFLV2014,DFLV2015} for $q=2$.

\begin{theorem}\label{thm:better}
For root functions of the form $f(w):=w^p$, with $p:=1/q$,
(\ref{star}) holds for integers $2\leq q \leq 10,000$.
\end{theorem}

\begin{proof}

The function $(g-h)(w)$, which we wish to prove is non-negative on the interval $[0,\delta]$, is
\[
\frac{\delta^{p-3}}{2}(p^2-3p+2)w^3 -\delta^{p-2}(p^2-4p+3)w^2 +\frac{\delta^{p-1}}{2}(p^2-5p+6)w
-(w+\hat{\lambda})^p+\hat{\lambda}^p ,
\]
where the shift constant $\hat{\lambda}$ for which $h'(0)=g'(0)$ is
\[
\hat{\lambda}
= (f')^{-1}(g'(0))
= \delta \left(\frac{p^2-5 p+6}{2p}\right)^{\frac{1}{p-1}}.
\]

With a few substitutions, we simplify the function and express it in polynomial form.  For the first substitution, set $q := 1/p$, $\gamma := \delta^{1/q}$, and  $t: = \frac{w}{\gamma^q}$.  We obtain a function of $t$ over $[0,1]$ that has $\gamma$ as a factor:
\[
\begin{array}{lcl}
(g-h)_t(t) &=& \gamma \left[ \frac{2 q^2-3 q+1}{2 q^2} t^3+\frac{-6 q^2+8 q-2}{2 q^2} t^2+\frac{6 q^2-5 q+1}{2 q^2} t \right.\\
 &&\left. -\left( \left(\frac{2q}{6 q^2-5 q+1}\right)^{\frac{q}{q-1}}+t\right)^{1/q}+\left(\frac{2q}{6 q^2-5 q+1}\right)^{\frac{1}{q-1}} \right].\\
 \\
\end{array}
\]
Next, we set $Q :=\left(\frac{2q}{6 q^2-5 q+1}\right)^{\frac{1}{q-1}}$ and $u :=(t+Q^q)^{1/q}$ (so $t \rightarrow u^q-Q^q$).  The resulting polynomial in $u$ is
\[
\begin{array}{lcl}
(g-h)_u(u) &= \frac{\gamma}{2q^2} &  \left[
\left(2 q^2-3 q+1\right) \left(u^q-Q^q\right)^3+\left(-6 q^2+8 q-2\right) \left(u^q-Q^q\right)^2 \right. \\
&& \left. +\left(6 q^2-5 q+1\right) \left(u^q-Q^q\right)+2q^2Q-2q^2u \right],
\end{array}
\]
for $Q \leq u \leq (1+Q^q)^{1/q}.$  Since $\gamma >0$, our task is reduced to proving that the second factor,
\[
K_u(u) := d \left(u^q-Q^q\right)^3-c\left(u^q-Q^q\right)^2 +b \left(u^q-Q^q\right)-a(u-Q), \\
\]
where
\begin{eqnarray*}
a &:=& 2q^2,\\
b &:=& 6q^2-5q+1,\\
c &:=& 6q^2-8q+2, \text{ and}\\
d &:=& 2q^2-3q+1,
\end{eqnarray*}
is non-negative for $Q \leq u \leq (1+Q^q)^{1/q}.$

It is obvious that $K_u$ has a root at $Q$.  In fact, $K_u$ has a double root at $Q$, which we can verify by showing that the first derivative of $K_u$,
\begin{eqnarray*}
K_u'(u)=3 d q  \left(u^q-Q^q\right)^2 u^{q-1}-2 c q  \left(u^q-Q^q\right)u^{q-1}+b q u^{q-1}-a,
\end{eqnarray*}
also has a root at $Q$.  This is easily accomplished by noticing that
\begin{eqnarray*}
b q u^{q-1} &=& b q \left(\left(\frac{2q}{6 q^2-5 q+1}\right)^{\frac{1}{q-1}}\right)^{q-1} = 2q^2 = a.
\end{eqnarray*}

By construction of $g$ and $h$,
\[
(g-h)_u((1+Q^q)^{1/q})=(g-h)(\delta)=(f-h)(\delta)>0,
\]
which means that $K_u((1+Q^q)^{1/q})>0$.  In order to prove that $K_u(u)\geq 0$ for $u \in (Q,(1+Q^q)^{1/q})$, it suffices to show that there are no roots in the interval $(Q,(1+Q^q)^{1/q})$.  In fact, we prove that the only root in the interval $(0,(1+Q^q)^{1/q})$ $\supseteq (Q,(1+Q^q)^{1/q})$ is the double root at $Q$.

Using a known technique (e.g., see \cite{Sagraloff}),
we apply the M\"{o}bius transformation
\[K_u\left(\frac{(1+Q^q)^{1/q}}{v+1}\right)\]
to express $K_u$, $u \in (0,(1+Q^q)^{1/q}))$, as a rational function in $v$ over the interval $(0, \infty)$.
Note that when $v=0$, $K_u\left(\frac{(1+Q^q)^{1/q}}{v+1}\right)=K_u\left((1+Q^q)^{1/q}\right),$ and as $v \rightarrow \infty$, $K_u\left(\frac{(1+Q^q)^{1/q}}{v+1}\right)\rightarrow K_u(0)$.

 Next, we calculate expressions for the coefficients of the polynomial
 \[
 K_v(v):=(v+1)^{3q} K_u\left(\frac{\beta}{v+1}\right),
 \]
where $\beta:=(1+Q^q)^{1/q}$, and the domain is $(0,\infty)$.

Expanding the binomials in $u^q$ and $Q^q$ and multiplying by $(v+1)^{3q}$, we have
\begin{eqnarray*}
&&K_v(v) =  (aQ - bQ^q -cQ^{2q}-dQ^{3q})(v+1)^{3q}  -a \beta (v+1)^{3q-1} \\
&& ~+(3d\beta^q Q^{2q} +2c\beta^q Q^q + b \beta^q)(v+1)^{2q}  -(3d\beta^{2q}Q^q + c\beta^{2q})(v+1)^q + d\beta^{3q}.
\end{eqnarray*}
Expanding binomials and collecting like terms, we find that
\begin{eqnarray*}
&& K_v(v)  =  V + W + X + Y + Z \\
&&~  + \sum_{i=1}^q \left[{q \choose q-i}W + {2q \choose 2q-i}X + {3q-1 \choose 3q-1-i}Y + {3q \choose 3q-i}Z\right]v^i \\
&& ~  + \sum_{i=q+1}^{2q} \left[{2q \choose 2q-i}X + {3q-1 \choose 3q-1-i}Y+ {3q \choose 3q-i}Z\right]v^i\\
&&~   + \sum_{i=2q+1}^{3q-1} \left[{3q-1 \choose 3q-1-i}Y + {3q \choose 3q-1}Z\right] v^i
~+~ Z v^{3q}, \\
\end{eqnarray*}
where
\begin{eqnarray*}
V & := & d\beta^{3q}, \\
W &:=&  -3d\beta^{2q}Q^q - c\beta^{2q},\\
X &:=& 3d\beta^q Q^{2q} +2c\beta^q Q^q + b \beta^q,\\
Y &:=& -a \beta, \text{ and}\\
Z &:=& -d Q^{3q} -c Q^{2q} -b Q^q + aQ.
\end{eqnarray*}

Armed with these expressions for the coefficients of the polynomials $K_v(v)$,
 we verified (with \verb;Mathematica;) that for
integers $2 \leq q \leq 10,000$,
there are exactly two sign changes in each coefficient sequence.  By Descartes' Rule of Signs, we conclude that there are at most two positive roots of $K_v(v)$ (for these values of $q$), and therefore at most two roots of $K_u(u)$ in the interval $(0, (1+Q^q)^{1/q})$ (the double root at $Q$).
\qed
\end{proof}

Our proof technique can work for any \emph{fixed} integer $q \geq 2$.
In carrying out the technique, there is some computational burden
for which we employ \verb;Mathematica;. We only carried this out for integers $2 \leq q \leq 10,000$, but in principle we could go further.
It is important to point
out that the calculations were done exactly and only truncated to
finite precision at the end.

The remaining challenge is to make a proof for all integers $q \geq 2$.
But the coefficients of $K_u\left(\frac{(1+Q^q)^{1/q}}{v+1}\right)$ are rather complicated for general $q$, so it is difficult to analyze their signs in general.

\section{Average performance for roots}\label{sec:avg}

On $[\delta,+\infty)$, $g$ coincides with $f$ by definition
(we are assuming that $f'(\delta)$ and $f''(\delta)$ are defined
so that $g$ is well defined),
while $h$ strictly under-estimates $f$ for increasing concave functions $f$ for which
$f(\delta)>0$. So it becomes interesting to
examine the performance of $g$ and $h$ near 0,
that is on the interval $[0,\delta]$. Here, we focus on
average relative performance:
\[
\frac{1}{\delta}\int_0^\delta \frac{g(w)}{f(w)} dw~,\qquad
\frac{1}{\delta}\int_0^\delta \frac{h(w)}{f(w)} dw~,
\]
 where we
are further assuming that $f(0)=0$ and $f$ is increasing.
In what follows, we compare these performance measures
for root functions.

\begin{thm}\label{thm:indep}
For $f(w) = w^p$, $0<p<1$, and $\lambda$ chosen as a function of $\delta >0$ so that $g'(0) = h'(0)$, we have that
\[
\frac{1}{\delta}\int_0^\delta \frac{g(w)}{f(w)} dw ~=~ \frac{3}{4-p},
\]
notably independent of $\delta$, and also
\[
\frac{1}{\delta}\int_0^\delta \frac{h(w)}{f(w)} dw
\]
is independent of the choice of $\delta$.
\end{thm}

\begin{proof}
We apply the change of variable $w \rightarrow \delta v$ to rewrite each expression without $\delta$.  For the first integral, letting $a := \frac{1}{2}(p^2 -3p+2)$, $b:= p^2-4p+3$, and $c := \frac{1}{2}(p^2-5p+6)$, we have
\begin{align*}
\frac{1}{\delta}\int_0^\delta \frac{g(w)}{f(w)} dw &~= \frac{1}{\delta} \int_0^\delta \frac{\delta^{p-3}aw^3 - \delta^{p-2}bw^2 + \delta^{p-1}cw}{w^p} dw  \\
&~= \int_0^1 \left(av^{3-p}-bv^{2-p} +cv^{1-p}\right) dv.  \\
&~= \frac{a}{4-p} -\frac{b}{3-p} +\frac{c}{2-p} ~=~ \frac{3}{4-p}~.
\end{align*}
Letting $P := \left(\frac{p^2-5 p+6}{2p}\right)^{\frac{1}{p-1}}$ in the second integral, the change of variable produces
\begin{align*}
\frac{1}{\delta}\int_0^\delta \frac{h(w)}{f(w)} dw &~= \frac{1}{\delta} \int_0^\delta \frac{(w+\delta P)^p-(\delta P)^p}{w^p} dw  \\
&~= \int_0^1 \frac{(v+P)^p - P^p}{v^p} dv,
\end{align*}
and again $\delta$ disappears from the expression.
\qed
\end{proof}

We note that $\frac{3}{4-p}$ is increasing in $p$, and so its infimum
on $(0,1)$ is $\frac{3}{4}$ at $p=0$.
Additionally, we note that
there is no closed-form expression  for the last integration of the proof, though it can be expressed in terms of an evaluation of a Gaussian/ordinary hypergeometric function $\tensor*[_2]{\rm F}{_1}$ (see \cite{andrews1999}).
Specifically
\[
\left(\frac{p^2-5 p+6}{2p}\right)^{\frac{p}{p-1}}
\left( \frac{-1+\tensor*[_2]{\rm F}{_1}(1-p\, ,\, -p\, ;\, 2-p\, ;\, -1/p)}{1-p}\right)~.
\]
In the key special case of $p=1/2$, we do get the closed-form expression \[
\frac{8}{15}\left(
-1+
\frac{3615+16\sqrt{241}\,{\rm ArcSinh}\!\left(\frac{15}{4}\right)}{120\sqrt{241}}
\right)\approx 0.646125
\]
(where ${\rm ArcSinh}(x)=\ln\left(x+\sqrt{x^2+1}\right)$),
which  is significantly less than
$\left.\frac{3}{4-p}\right\rvert_{p=\frac{1}{2}}= \frac{6}{7} \approx 0.857143$.

In Figure \ref{fig:averagefor roots}
we have plotted the two performance measures,
varying $p$ on $(0,1)$; Theorem \ref{thm:indep}
allows us to do this without separate curves for different $\delta$.
We can readily see that $g$ outperforms $h$ bigly,
with the performance gap being
most extreme as $p\to 0$, and decreasing in $p$
on $(0,1)$.

\begin{figure}[ht]
\begin{center}
\includegraphics[width=.95\textwidth]{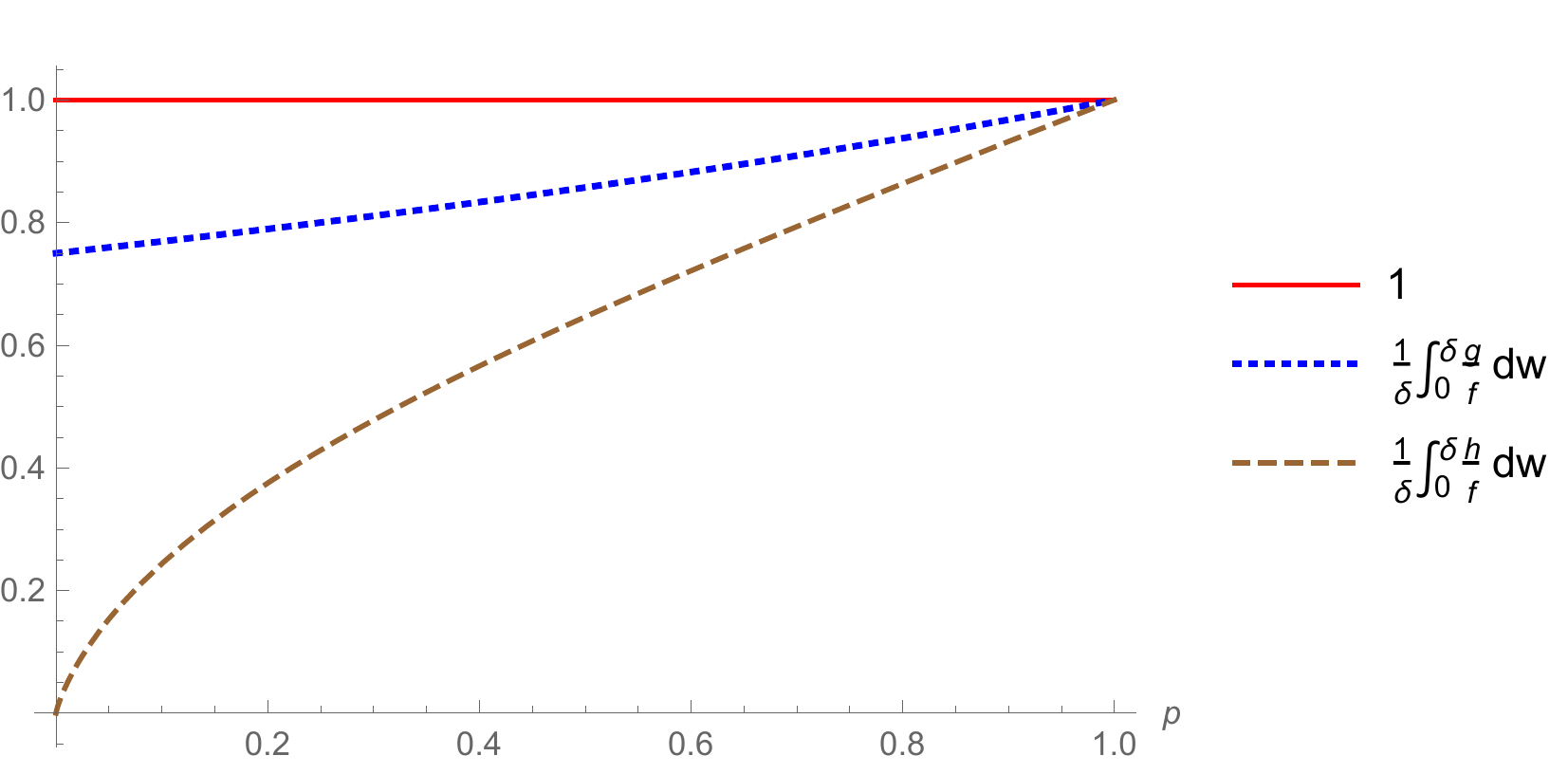}
\captionsetup{justification=centering}
\caption{Average relative performance on $[0,\delta]$ as a function of $p$\\ ($f(w):=w^p$,~ $0<p<1$)}
\label{fig:averagefor roots}
\end{center}
\end{figure}

\section{Alternatives, software, extended use, and ongoing work}\label{sec:future}

\subsection{Alternatives}
We do not mean to imply that our smoothing ideas are the only viable or even preferred way
to handle all instances of the functions to which our ideas apply. For example, there
is the possibility to take a function $f(w)$ and replace it with a new variable $y$ and
the constraint $f^{-1}(y)=w$. For example, $f(w):=w^p$, with $0<p<1$, can be replaced
with $y$ and the constraint $w=y^{\frac{1}{p}}$, which is now smooth at $w=0$.
But there is the computational cost of including an additional non-linear equation in a model.
In fact, this could turn an unconstrained model into a constrained model. For other functions,
we may not have a nice expression for the inverse readily available.
Even when we do, there can be other issues to consider.
Taking again $f(w):=w^p$, with $0<p<1$,
suppose that we have $\psi:\mathbb{R}^m\rightarrow \mathbb{R}_+$
and $\phi:\mathbb{R}^n\rightarrow \mathbb{R}$. Now suppose that
we have a model with the constraint $f(\psi(y))\geq \phi(x)$.
Of course $y$ and $x$ may be involved in other constraints as well.
Now, suppose further that the range of $\phi$ on the feasible region
is $(-\infty,+\infty)$. There could then be a difficulty
in trying to work with $\psi(y) \geq f^{-1}(\phi(x))$, which we could repair by
instead working with $\psi(y) \geq f^{-1}(\phi_+(x))$, where $\phi_+(x):=\max\{\phi(x),0\}$.
But this means working now with the piecewise-defined function $\phi_+$, which can involve
a lot of pieces (e.g., consider the univariate $\phi(x):=\sin(x)$).

In the end, we do not see our technique as a panacea, but
rather as a viable method with nice properties that modelers and solvers should have in their
bags.

\subsection{Software}\label{subsec:soft}

In the context of the square-root smoothing of \cite{DFLV2014,DFLV2015},
a new and exciting experimental (``$\alpha$'') feature was developed for \verb;SCIP; Version 3.2
(see \cite{SCIP32})
to handle univariate piecewise
functions that are user-certified (through \verb;AMPL; suffixes)
as being globally convex or concave.  At this writing,
\verb;SCIP; is the only global solver that can accommodate
such functions.
Such a feature is extremely useful for taking advantage of the results that we present here, because our smoothings (like those of \cite{DFLV2014,DFLV2015}) are piecewise-defined, and so the user
must identify global concavity to the solver.
This is accomplished through the new
\verb;SCIP; operator type \verb;SCIP_EXPR_USER;. A bit of detail about this feature, from \cite{SCIP32},
is enlightening:

\begin{quote}
``Currently, the following callbacks can or have to be provided:
computing the value, gradient, and Hessian of a user-function, given values (numbers or
intervals) for all arguments of the function; indicating convexity/concavity based on bounds
and convexity information of the arguments; tighten bounds on the arguments of the function
when bounds on the function itself are given; estimating the function from below or above by
a linear function in the arguments; copy and free the function data. Currently, this feature is
meant for low-dimensional (few arguments) functions that are fast to evaluate and for which
bound propagation and/or convexification methods are available that provide a performance
improvement over the existing expression framework.''
\end{quote}

\subsection{Extended use}
Our techniques have broader applicability than to functions that are purely concave (or symmetrically, convex). In general, given a univariate piecewise-defined function
that is concave or convex on each piece, and possibly non-differentiable
at each breakpoint, we can seek to find a smoothing that closely mimics and approximates the function. It is not at all clear how to accommodate such functions in global-optimization software like \verb;SCIP;; because such functions are not, in general, \emph{globally} concave or convex,
they cannot be correctly handled with the new \verb;SCIP; feature (see \S\ref{subsec:soft}). Still, such functions and their smoothings can be useful within the common paradigm of seeking (good) local optima for non-linear-optimization formulations.

For example, for $0<p<1$, consider the function
\[
f(w):= \left\{
         \begin{array}{ll}
           w^p,\quad & w\geq 0; \\
           -(-w)^p,\quad & w\leq 0.
         \end{array}
       \right.
\]
This function is continuous, increasing, convex on $(-\infty,0]$,
concave on $[0,+\infty)$ and of course not differentiable at 0.
We would like to replace it with a function that has
all of these properties but is somewhat smooth at 0.
If we apply our smoothing to $f$ separately, for
$w\geq 0$ and for $w\leq 0$, we would arrive at
a function of the form
\[
g(w):= \left\{
         \begin{array}{ll}
           w^p,\quad & w\geq \delta; \\
           Aw^3+Bw^2+Cw,\quad & 0\leq w \leq \delta;\\
           Aw^3-Bw^2+Cw,\quad  & -\delta \leq w \leq 0;\\
            -(-w)^p,\quad & w\leq -\delta,
         \end{array}
       \right.
\]
for appropriate $A,B,C$ (see \S\ref{sec:general}).
It is easy to check that the resulting $g$ is continuous, \emph{differentiable at 0}, twice differentiable
everywhere but at 0, increasing, convex on $(-\infty,0]$, and
concave on $[0,+\infty)$. In short, $g$ mimics $f$ very well, but is smoother.
Moreover, for $p=1/q$, with integer $q\geq 2$, $g$ upper bounds $f$ on $(-\infty,0]$ and lower bounds $f$ on $[0,+\infty)$.

Note that the obvious ``double shift''
\[
h(w):= \left\{
         \begin{array}{ll}
           (w+\lambda)^p,\quad & w> 0; \\
           ?, \quad & w= 0; \\
           -(-w+\lambda)^p,\quad & w< 0
         \end{array}
       \right.
\]
is not even continuous at 0.

Additionally, for $p=1/q$, with integer $2\leq q\leq 10,000$, in the sense of
\S\ref{sec:better},
$g$ is a better upper bound
on $f$ than $h$ on $(-\infty,0]$,
and $g$ is a better lower bound
on $f$ than $h$ on
 $[0,+\infty)$.

\subsection{Ongoing work}

To extend the applicability of our results, we are pursuing two
directions:
\begin{itemize}
\item We would like to generalize Theorem \ref{thm:lower} to all root functions
$w^p$ with $0<p<1$.
A strategy that we are exploring is to try to make a
 similar proof to what we have, for the case in which $p$ is rational, and then employ a continuity argument to establish the result for all real exponents.
\item  We would like to generalize Theorem \ref{thm:better} for all root functions
$w^p$ with $0<p<1$. For now, that seems like a rather ambitious goal, and
what is more in sight is  generalizing Theorem \ref{thm:better}
for \emph{all} integer $q \geq 2$. To do this we are trying to sharpen
our arguments employing Descartes' Rule of Signs, or, alternatively, to
develop a sum-of-squares argument.
\end{itemize}

\begin{acknowledgements}
The authors gratefully acknowledge the anonymous referee
who proposed the performance measure studied in \S\ref{sec:avg}.
J. Lee gratefully acknowledges partial support from
ONR grant N00014-14-1-0315.
\end{acknowledgements}

\bibliographystyle{amsplain}
\bibliography{smooth}

\providecommand{\bysame}{\leavevmode\hbox to3em{\hrulefill}\thinspace}
\providecommand{\MR}{\relax\ifhmode\unskip\space\fi MR }
\providecommand{\MRhref}[2]{%
  \href{http://www.ams.org/mathscinet-getitem?mr=#1}{#2}
}
\providecommand{\href}[2]{#2}
\begin{thebibliography}{10}

\bibitem{Achterberg2009}
Tobias Achterberg, \emph{{SCIP}: Solving constraint integer programs},
  Mathematical Programming Computation \textbf{1} (2009), no.~1, 1--41.

\bibitem{andrews1999}
George~E. Andrews, Richard Askey, and Ranjan Roy, \emph{Special functions:},
  Cambridge University Press, 1999.

\bibitem{arrow}
Kenneth~J. Arrow, Hollis~B. Chenery, Bagicha~S. Minhas, and Robert~M. Solow,
  \emph{Capital-labor substitution and economic efficiency}, The Review of
  Economics and Statistics \textbf{43} (1961), no.~3, 225--250.

\bibitem{Belotti09}
Pietro Belotti, Jon Lee, Leo Liberti, Fran\c{c}ois Margot, and Andreas
  W\"{a}chter, \emph{Branching and bounds tightening techniques for non-convex
  {MINLP}}, Optimizaton Methods \& Software \textbf{24} (2009), no.~4--5,
  597--634.

\bibitem{Bonami}
Pierre Bonami, Lorenz~T. Biegler, Andrew~R. Conn, G{\'e}rard Cornu{\'e}jols,
  Ignacio~E. Grossmann, Carl~D. Laird, Jon Lee, Andrea Lodi, Fran{\c{c}}ois
  Margot, Nicolas Sawaya, and Andreas W{\"a}chter, \emph{An algorithmic
  framework for convex mixed integer nonlinear programs}, Discrete Optimization
  \textbf{5} (2008), no.~2, 186--204.

\bibitem{BLLW}
Pierre Bonami, Jon Lee, Sven Leyffer, and Andreas W{\"a}chter, \emph{On
  branching rules for convex mixed-integer nonlinear optimization}, ACM Journal
  of Experimental Algorithmics \textbf{18} (2013), Article 2.6, 31 pages.

\bibitem{BDLLT06}
Cristiana Bragalli, Claudia D'Ambrosio, Jon Lee, Andrea Lodi, and Paolo Toth,
  \emph{An {MINLP} solution method for a water network problem},
  Algorithms---{ESA} 2006, Lecture Notes in Computer Science, vol. 4168,
  Springer, Berlin, 2006, pp.~696--707.

\bibitem{BDLLT12}
\bysame, \emph{On the optimal design of water distribution networks},
  Optimization and Engineering \textbf{13} (2012), no.~2, 219--246.

\bibitem{Cozad}
Alison Cozad, Nikolaos~V. Sahinidis, and David~C. Miller, \emph{A combined
  first-principles and data-driven approach to model building}, Computers \&
  Chemical Engineering \textbf{73} (2015), 116--127.

\bibitem{DFLV2014}
Claudia D'Ambrosio, Marcia Fampa, Jon Lee, and Stefan Vigerske, \emph{On a
  nonconvex {MINLP} formulation of the {E}uclidean {S}teiner tree problems in
  n-space}, Tech. report, Optimization Online, 2014,
  \url{http://www.optimization-online.org/DB_HTML/2014/09/4528.html}.

\bibitem{DFLV2015}
\bysame, \emph{On a nonconvex {MINLP} formulation of the {E}uclidean {S}teiner
  tree problem in n-space}, Experimental Algorithms (E.~Bampis, ed.), Lecture
  Notes in Computer Science, vol. 9125, Springer International Publishing,
  2015, pp.~122--133.

\bibitem{douglas}
Paul~H. Douglas, \emph{The {C}obb-{D}ouglas production function once again: Its
  history, its testing, and some new empirical values}, Journal of Political
  Economy \textbf{84} (1976), no.~5, 903--915.

\bibitem{Duffin1973}
Richard~J. Duffin and Elmor~L. Peterson, \emph{Geometric programming with
  signomials}, Journal of Optimization Theory and Applications \textbf{11}
  (1973), no.~1, 3--35.

\bibitem{ITOR:ITOR12207}
Marcia Fampa, Jon Lee, and Nelson Maculan, \emph{{An overview of exact
  algorithms for the Euclidean Steiner tree problem in n-space}}, International
  Transactions in Operational Research \textbf{23} (2016), no.~5, 861--874.

\bibitem{SCIP32}
Tristan Gally, Ambros~M. Gleixner, Gregor Hendel, Thorsten Koch, Stephen~J.
  Maher, Matthias Miltenberger, Benjamin M\"uller, Marc~E. Pfetsch, Christian
  Puchert, Daniel Rehfeldt, Sebastian Schenker, Robert Schwarz, Felipe Serrano,
  Yuji Shinano, Stefan Vigerske, Dieter Weninger, Michael Winkler, Jonas~T.
  Witt, and Jakob Witzig, \emph{{The SCIP Optimization Suite 3.2}}, February
  2016, ZR 15-60, Zuse Institute Berlin.
  \url{http://www.optimization-online.org/DB_HTML/2016/03/5360.html}.

\bibitem{GMS13}
Iacopo Gentilini, Fran\c{c}ois Margot, and Kenji Shimada, \emph{The travelling
  salesman problem with neighbourhoods: {MINLP} solution}, Optimization Methods
  and Software \textbf{28} (2013), no.~2, 364--378.

\bibitem{Wotao}
Ming-Jun Lai, Yangyang Xu, and Wotao Yin, \emph{Improved iteratively reweighted
  least squares for unconstrained smoothed $\ell_q$ minimization.}, SIAM J.
  Numerical Analysis \textbf{51} (2013), no.~2, 927--957.

\bibitem{misener-floudas:ANTIGONE:2014}
Ruth Misener and Christodoulos~A. Floudas, \emph{{ANTIGONE: Algorithms for
  coNTinuous / Integer Global Optimization of Nonlinear Equations}}, Journal of
  Global Optimization (2014), 503--526.

\bibitem{Sagraloff}
Michael Sagraloff, \emph{On the complexity of the {D}escartes method when using
  approximate arithmetic}, Journal of Symbolic Computation \textbf{65} (2014),
  79--110.

\bibitem{Sergeyev1998}
Yaroslav~D. Sergeyev, \emph{Global one-dimensional optimization using smooth
  auxiliary functions}, Mathematical Programming \textbf{81} (1998), no.~1,
  127--146.

\bibitem{Smith99}
Edward~M.B. Smith and Constantinos~C. Pantelides, \emph{A symbolic
  reformulation/spatial branch-and-bound algorithm for the global optimisation
  of nonconvex {MINLP}s}, Computers \& Chemical Engineering \textbf{23} (1999),
  457--478.

\bibitem{TawarSahinBook}
Mohit Tawarmalani and Nikolaos~V. Sahinidis, \emph{Convexification and global
  optimization in continuous and mixed-integer nonlinear programming},
  Nonconvex Optimization and its Applications, vol.~65, Kluwer Academic
  Publishers, Dordrecht, 2002, Theory, algorithms, software, and applications.

\bibitem{sahinidis}
\bysame, \emph{Convexification and global optimization in continuous and
  mixed-integer nonlinear programming: Theory, algorithms, software, and
  applications}, Nonconvex Optimization and Its Applications, Springer US,
  2002.

\bibitem{WB06}
Andreas W{\"a}chter and Lorenz~T. Biegler, \emph{On the implementation of an
  interior-point filter line-search algorithm for large-scale {NLP}},
  Mathematical Programming, Series A \textbf{106} (2006), 25--57.

\end{thebibliography}

\newpage

%
%
%

%

\section*{Appendix}

\begin{lem}\label{lem:terms}
The polynomial $P_q$ as defined above for integer $q \geq 2$ can be expressed as $Q_q (L-t)^3$, where polynomial $Q_q$ has the following $3q-3$ terms:
\begin{eqnarray*}
{i+2 \choose 2}a L^{3q-4-i} t^{i}, &\mbox{for}& i = 0,1, \dots, q-2;\\
\left[ {i+2 \choose 2}a - {i-q+3 \choose 2}b \right] L^{3q-4-i} t^{i}, &\mbox{for}& i = q-1,q, \dots, 2q-2;\\
\left[{i+2 \choose 2}a - {i-q+3 \choose 2}b +{i-2q+3 \choose 2}c \right] L^{3q-4-i} t^{i}, &\mbox{for}& i = 2q-1,2q, \dots, 3q-4.
\end{eqnarray*}
Note that for $q=2$, $2q-2 = 3q-4 = 2$, so there are no terms of the third type.
\end{lem}


\begin{proof}
We expand $Q_q (L-t)^3$ to see that it is equivalent to $P_q$, first for specific cases $q = 2, 3 \mbox{ and } 4$, and finally for general $q \geq 5$.  The following
are easily verified:
\begin{eqnarray*}
 &&\underline{Q_q (L-t)^3 = P_q}\\
q=2:&&(8L^2+9Lt+3t^2) (L-t)^3 = 8L^5-15L^4t+10L^2t^3-3t^5;\\
q=3:&&(18L^5+54L^4t+68L^3t^2+60L^2t^3+30Lt^4+10t^5) (L-t)^3 \\
&&\quad = 18L^8-40L^6t^2+32L^3t^5-10t^8;\\
q=4:&& (32L^8+96L^7 t+192L^6t^2+243L^5t^3+249L^4t^4+210L^3t^5  \\
&&\quad  +126L^2t^6+63L t^7+21t^8) (L-t)^3  = 32L^{11}-77L^8t^3+66L^4t^7-21t^{11}.
\end{eqnarray*}
Now considering general $q \geq 5$, most of the terms of $Q_q(L-t)^3$ cancel out due to the following equations:
\begin{eqnarray*}
(-3){2 \choose 2} + {3 \choose 2} &= -3+3&= 0;\\
(3){2 \choose 2} +(-3){3 \choose 2} + {4 \choose 2}&=3-9+6&= 0; \\
\end{eqnarray*}
and for $i \geq 3,$
\begin{eqnarray*}
(-1){i-1 \choose 2} + (3){i \choose 2} +(-3){i+1 \choose 2} + {i+2 \choose 2}&&=0.\\
\end{eqnarray*}
If the expression for the coefficient of $t^j$, $0\leq j \leq 3q-1$, has more than one term involving $a$, $b$, or $c$, that variable ($a$, $b$, or $c$) has a coefficient of one of the forms above and cancels out.  The only time the variable remains is when it has only a single term in the expression.  The terms of $Q_q(L-t)^3 = Q_q(-t^3+3Lt^2-3L^2t+L^3)$ for $q \geq 5$ increasing in the degree $j$ of $t$ are as follows:
\begin{eqnarray*}
j=0:&& \left(aL^{3q-4}\right)\left(L^3\right)=aL^{3q-1} \\
j=1:&& \left(aL^{3q-4}\right)\left(-3L^2t\right) + \left(3aL^{3q-5}t\right)\left(L^3\right) \\
&&  =(-3 + 3)aL^{3q-2}t \\
&&  = 0\\
j=2:&& \left(aL^{3q-4}\right)\left(3Lt^2\right) + \left(3aL^{3q-5}t\right)\left(-3L^2t\right) + \left(6aL^{3q-6}t^2\right)\left(L^3\right)\\
&&  = (3-9+6)aL^{3q-3}t^2\\
&&  = 0 \\
3 \leq j \leq q-2:&& {j-1 \choose 2}aL^{3q-1-j}t^{j-3}\left(-t^3\right) + {j \choose 2}aL^{3q-2-j}t^{j-2}\left(3Lt^2\right) + \\
&&\quad {j+1 \choose 2}aL^{3q-3-j}t^{j-1}\left(-3L^2t\right) + {j+2 \choose 2}aL^{3q-4-j}t^{j}\left(L^3\right)\\
&& = \left[(-1){j-1 \choose 2} + (3){j \choose 2} +(-3){j+1 \choose 2} + {j+2 \choose 2}\right]aL^{3q-1-j}t^j\\
&&= 0\\
j=q-1:&& {q-2 \choose 2}aL^{2q}t^{q-4}\left(-t^3\right) + {q-1 \choose 2}aL^{2q-1}t^{q-3}\left(3Lt^2\right) + \\
&&\quad {q \choose 2}aL^{2q-2}t^{q-2}\left(-3L^2t\right) + \left[{q+1 \choose 2}a-b\right]L^{2q-3}t^{q-1} \left(L^3\right)\\
&&= \left[(-1){q-2 \choose 2} + (3){q-1 \choose 2} +(-3){q \choose 2} + {q+1 \choose 2}\right]aL^{2q}t^{q-1} \\
&& \quad -bL^{2q}t^{q-1}\\
&&= -bL^{2q}t^{q-1}\\
j=q: &&{q-1 \choose 2}aL^{2q-1}t^{q-3}\left(-t^3\right) + {q \choose 2}aL^{2q-2}t^{q-2}\left(3Lt^2\right) + \\
&&\quad \left[{q+1 \choose 2}a-b\right]L^{2q-3}t^{q-1}\left(-3L^2t\right) + \left[{q+2 \choose 2}a-3b\right]L^{2q-4}t^{q} \left(L^3\right)\\
&& =\left[(-1){q-1 \choose 2} + (3){q \choose 2} +(-3){q+1 \choose 2} + {q+2 \choose 2}\right]aL^{2q-1}t^{q} \\
&& \quad +(3-3)bL^{2q-1}t^{q}\\
&& = 0\\
j=q+1: &&{q \choose 2}aL^{2q-2}t^{q-2}\left(-t^3\right) \\
&&\quad + \left[{q+1 \choose 2}a-b\right]L^{2q-3}t^{q-1}\left(3Lt^2\right) \\
&&\quad  + \left[{q+2 \choose 2}a-3b\right]L^{2q-4}t^{q}\left(-3L^2t\right) \\
&&\quad + \left[{q+3 \choose 2}a-6b\right]L^{2q-5}t^{q+1} \left(L^3\right) \\
&& = 0\\
q+2 \leq j \leq 2q-2: &&\left[{j-1 \choose 2}a-{j-q \choose 2}b\right]L^{3q-1-j}t^{j-3}\left(-t^3\right) \\
&& \quad + \left[{j \choose 2}a-{j-q+1 \choose 2}b\right]L^{3q-2-j}t^{j-2}\left(3Lt^2\right) \\
&&\quad + \left[{j+1 \choose 2}a-{j-q+2 \choose 2}b\right]L^{3q-3-j}t^{j-1}\left(-3L^2t\right) \\
&&\quad + \left[{j+2 \choose 2}a-{j-q+3 \choose 2}b\right]L^{3q-4-j}t^j \left(L^3\right)\\
&& = 0\\
j=2q-1:&&\left[{2q-2 \choose 2}a-{q-1 \choose 2}b\right]L^{q}t^{2q-4}\left(-t^3\right) \\
&& \quad + \left[{2q-1 \choose 2}a-{q \choose 2}b\right]L^{q-1}t^{2q-3}\left(3Lt^2\right) \\
&&\quad + \left[{2q \choose 2}a-{q+1 \choose 2}b\right]L^{q-2}t^{2q-2}\left(-3L^2t\right) \\
&&\quad + \left[{2q+1 \choose 2}a-{q+2 \choose 2}b+c\right]L^{q-3}t^{2q-1} \left(L^3\right)\\
&& = cL^qt^{2q-1} \\
j=2q:&&\left[{2q-1 \choose 2}a-{q \choose 2}b\right]L^{q-1}t^{2q-3}\left(-t^3\right) \\
&& \quad + \left[{2q \choose 2}a-{q+1 \choose 2}b\right]L^{q-2}t^{2q-2}\left(3Lt^2\right) \\
&&\quad + \left[{2q+1 \choose 2}a-{q+2 \choose 2}b+c\right]L^{q-3}t^{2q-1}\left(-3L^2t\right) \\
&&\quad + \left[{2q+2 \choose 2}a-{q+3 \choose 2}b+3c\right]L^{q-4}t^{2q} \left(L^3\right)\\
&& = 0 \\
j=2q+1:&&\left[{2q \choose 2}a-{q+1 \choose 2}b\right]L^{q-2}t^{2q-2}\left(-t^3\right) \\
&& \quad + \left[{2q+1 \choose 2}a-{q+2 \choose 2}b+c\right]L^{q-3}t^{2q-1}\left(3Lt^2\right) \\
&&\quad + \left[{2q+2 \choose 2}a-{q+3 \choose 2}b+3c\right]L^{q-4}t^{2q}\left(-3L^2t\right) \\
&&\quad + \left[{2q+3 \choose 2}a-{q+4 \choose 2}b+6c\right]L^{q-5}t^{2q+1} \left(L^3\right)\\
&& = 0 \\
2q+2 \leq j \leq 3q-4:&&\left[{j-1 \choose 2}a-{j-q \choose 2}b+{j-2q \choose 2}c\right]L^{3q-1-j}t^{j-3}\left(-t^3\right) \\
&& \quad + \left[{j \choose 2}a-{j-q+1 \choose 2}b+{j-2q+1 \choose 2}c\right]L^{3q-2-j}t^{j-2}\left(3Lt^2\right) \\
&&\quad + \left[{j+1 \choose 2}a-{j-q+2 \choose 2}b+{j-2q+2 \choose 2}c\right]L^{3q-3-j}t^{j-1}\left(-3L^2t\right) \\
&&\quad + \left[{j+2 \choose 2}a-{j-q+3 \choose 2}b+{j-2q+3 \choose 2}c\right]L^{3q-4-j}t^{j} \left(L^3\right)\\
&& = 0 \\
\end{eqnarray*}
The cancellation pattern above fails for the last three terms.  It is necessary to replace $a$, $b$, and $c$ with the equivalent expressions involving $q$ to verify each of the following.  
\begin{eqnarray*}
j=3q-3:&& \left[{3q-4 \choose 2}a - {2q-3 \choose 2}b + {q-3 \choose 2}c\right]L^2t^{3q-6}(-t^3) \\
&& \quad + \left[{3q-3 \choose 2}a - {2q-2 \choose 2}b + {q-2 \choose 2}c\right]Lt^{3q-5}(3Lt^2) \\
&& \quad + \left[{3q-2 \choose 2}a - {2q-1 \choose 2}b + {q-1 \choose 2}c\right]t^{3q-4}(-3L^2t) \\
&&  =0\\
j=3q-2:&&  \left[{3q-3 \choose 2}a - {2q-2 \choose 2}b + {q-2 \choose 2}c\right]Lt^{3q-5}(-t^3) \\
&& \quad + \left[{3q-2 \choose 2}a - {2q-1 \choose 2}b + {q-1 \choose 2}c\right]t^{3q-4}(3Lt^2) \\
&& =0\\
j=3q-1:&& \left[{3q-2 \choose 2}a - {2q-1 \choose 2}b + {q-1 \choose 2}c\right]t^{3q-4}(-t^3) =-dt^{3q-1}
\end{eqnarray*}
\qed
\end{proof}

\begin{lem}\label{app:lemma2}
The polynomial $Q_q$ as defined in the previous lemma for integers $q \geq 2$ has all positive coefficients.
\end{lem}

\begin{proof}
We consider each of the three types of coefficients of $Q_q$ separately.  The first type of coefficients,
\[
{i+1 \choose 2}2q^2, \mbox{ for } i = 1,2, \dots, q-1,
\]
are obviously all positive.

Coefficients of the second type have the form
\[
{i+1 \choose 2}a - {i-q+2 \choose 2}b, \mbox{ for } i = q, q+1, \dots, 2q-1.
\]
The real function $C_2(x) = \frac{1}{2}(x+1)(x)a - \frac{1}{2}(x-q+2)(x-q)b, ~x \in [q,2q-1]$, has second derivative $C_2''(x)=-4q^2+5q-1$, which is negative for $q>1$.  Therefore, $C_2$ is concave on the interval $[q,2q-1]$.  Evaluating $C_2$ at the ends of the interval, we find that $C_2(q)=q^4+q^3-6q^2+5q-1$ and $C_2(2q-1)=q^4-\frac{5}{2}q^3+2q^2-\frac{1}{2}q$, both of which are positive for $q>1$.  We conclude that $C_2$ is positive over the interval $[q,2q-1]$, and all of the type-two coefficients are positive.

Finally, the third type of coefficients have the form
\[
{i+1 \choose 2}a - {i-q+2 \choose 2}b +{i-2q+2 \choose 2}c, \mbox{ for } i = 2q, 2q+1, \dots, 3q-3.
\]
As above, we consider the real extension of this function, $C_3(x) = \frac{1}{2}(x+1)(x)a - \frac{1}{2}(x-q+2)(x-q+1)b +\frac{1}{2}(x-2q+2)(x-2q+1)c, ~x \in [2q,3q-3]$.  The first derivative of this function, $C_3'(x) = (2q^2-3q+1)x -(6q^3-12q^2+\frac{15}{2} q -\frac{3}{2})$, is linear in $x$ and has positive slope for $q>1$.  Furthermore, the $x$ intercept of $C_3'(x)$ is $x = 3q-\frac{3}{2}$.  This means that $C_3'(x)<0$ for $x < 3q-\frac{3}{2}$.  In particular, $C_3(x)$ is decreasing on the interval $[2q, 3q-3]$.  The right end of this interval is $C_3(3q-3)=2q^2-3q+1$, which is positive for $q>1$, and all of the type-three terms are positive.
\qed
\end{proof}

\end{document}